\documentclass[12pt]{amsart}\usepackage{amsmath,amssymb,amscd,verbatim,color,graphicx,enumerate}

\usepackage{hyperref}
\let\svthefootnote\thefootnote
\newtheorem{Proposition}{Proposition}[section]
\newtheorem{Theorem}[Proposition]{Theorem}
\newtheorem{Lemma}[Proposition]{Lemma}
\newtheorem{cor}[Proposition]{Corollary}
\newtheorem{Problem}[Proposition]{Problem}
\newtheorem{Conjecture}[Proposition]{Conjecture}
\newtheorem{remark}[Proposition]{Remark}
\numberwithin{equation}{section}\allowdisplaybreaks
\newcommand{\shift}{\operatorname{shift}}
\newcommand{\supp}{\mathrm{supp}}
\begin{document}
\title[Mean dimension and an embedding theorem for real flows]
{Mean dimension and an embedding theorem\\for real flows}
\author[Y. Gutman]{Yonatan Gutman}
\address{Yonatan Gutman:
Institute of Mathematics, Polish Academy of Sciences, ul. \'Sniadeckich 8, 00-656 Warszawa, Poland}
\email{y.gutman@impan.pl}
\author[L. Jin]{Lei Jin $^*$}\let\thefootnote\relax\footnote{* Corresponding author.}
\addtocounter{footnote}{-1}\let\thefootnote\svthefootnote
\address{Lei Jin:
Center for Mathematical Modeling, University of Chile and UMI 2807 - CNRS}
\email{jinleim@mail.ustc.edu.cn}
\subjclass[2010]{37B05; 54H20.}\keywords{Mean dimension, real flow, equivariant embedding, band-limited function}\begin{abstract}
We develop mean dimension theory for $\mathbb{R}$-flows.
We obtain fundamental properties and examples and prove an embedding theorem:
Any real flow $(X,\mathbb{R})$ of mean dimension strictly less than $r$ admits
an extension $(Y,\mathbb{R})$ whose mean dimension is equal to that of $(X,\mathbb{R})$
and such that $(Y,\mathbb{R})$ can be embedded in the $\mathbb{R}$-shift on the compact function space
$\{f\in C(\mathbb{R},[-1,1])|\;\supp(\hat{f})\subset [-r,r]\}$, where $\hat{f}$ is the Fourier transform of $f$ considered as a tempered distribution. These canonical embedding spaces appeared previously as a tool in embedding results for $\mathbb{Z}$-actions.
\end{abstract}\maketitle\section{Introduction}\label{sec:intro}
Mean dimension was introduced by Gromov \cite{Grom} in 1999,
and was systematically studied by Lindenstrauss and Weiss \cite{LW}
as an invariant of topological dynamical systems (t.d.s).
In recent years it has extensively been investigated with relation to the so-called
embedding problem, mainly for $\mathbb{Z}^k$-actions ($k\in\mathbb{N}$).
For $\mathbb{Z}$-actions, the problem is which $\mathbb{Z}$-actions $(X,T)$ can be embedded in the shifts on the Hilbert cubes
$(([0,1]^N)^\mathbb{Z},\sigma)$, where $N$ is a natural number and the shift $\sigma$ acts on $([0,1]^N)^\mathbb{Z}$ by
$\sigma((x_n)_{n\in\mathbb{Z}})=(x_{n+1})_{n\in\mathbb{Z}}$ for $x_n\in [0,1]^N$.
Under the conditions that $X$ has finite Lebesgue covering dimension and the system
$(X,T)$ is aperiodic, Jaworski \cite{J} proved in 1974 that
$(X,T)$ can be embedded in the shift on $[0,1]^\mathbb{Z}$.
Using Fourier and complex analysis, Gutman and Tsukamoto showed that
if $(X,T)$ is minimal and has mean dimension strictly less than $N/2$
then it can be embedded in $(([0,1]^N)^\mathbb{Z},\sigma)$ (see a more general result in \cite{GQT}).
We note that the value $N/2$ is optimal since
a minimal system of mean dimension
$N/2$ which cannot be embedded in $(([0,1]^N)^\mathbb{Z},\sigma)$
was constructed in \cite[Theorem 1.3]{LT}. More references for the embedding problem are given in
\cite{A,K,L,G11,G,GT14,GLT,G16,G17,GQS}.

In this paper, we develop the mean dimension theory for $\mathbb{R}$-actions
and investigate the embedding problem in this context.
Throughout this paper, by a {\bf flow}
we mean a pair $(X,\mathbb{R})$, where $X$ is a compact metric space
and $\Gamma:\mathbb{R}\times X\to X,(r,x)\mapsto rx$ is a continuous map
such that $\Gamma(0,x)=x$ and $\Gamma(r_1,\Gamma(r_2,x))=\Gamma(r_1+r_2,x)$
for all $r_1,r_2\in\mathbb{R}$ and $x\in X$.
Let $(X,\mathbb{R})=(X,(\varphi_r)_{r\in\mathbb{R}})$ and
$(Y,\mathbb{R})=(Y,(\phi_r)_{r\in\mathbb{R}})$ be flows.
We say that $(Y,\mathbb{R})$ can be {\bf embedded} in $(X,\mathbb{R})$ if there is
an $\mathbb{R}$-equivariant homeomorphism of $Y$ onto a subspace of $X$; namely,
there is a homeomorphism $f:Y\to f(Y)\subset X$ such that
$f\circ\phi_r=\varphi_r\circ f$ for all $r\in\mathbb{R}$.

This paper is organized as follows:
In Section \ref{sec:mdim}, we present basic notions and properties
of mean dimension theory for flows. In Section \ref{sec:construction} we construct minimal real flows with arbitrary mean dimension. In Section \ref{sec:conj}, we propose an embedding conjecture for flows and discuss its relation to the Lindenstrauss-Tsukamoto embedding conjecture for $\mathbb{Z}$-systems.
In Section \ref{sec:embedding}, we state the main embedding theorem
and prove it using a key proposition.
In Section \ref{sec:prod embed}, we prove the key proposition.

\bigskip

\noindent\textbf{Acknowledgements.}
Y. Gutman was partially supported by the NCN (National Science Center, Poland) Grant 2016/22/E/ST1/00448. Y. Gutman and L. Jin were partially supported by the NCN (National Science Center, Poland) Grant 2013/08/A/ST1/00275. L. Jin was supported by Basal funding PFB 170001 and Fondecyt Grant No. 3190127. This work owes greatly to previous work by Y. Gutman and M. Tsukamoto. We are grateful to the anonymous reviewer for a careful reading and many useful suggestions.

\bigskip

\section{Mean dimension for real flows}\label{sec:mdim}
We first introduce the definition of mean dimension for $\mathbb{R}$-actions.
Let $(X,d)$ be a compact metric space.
Let $\epsilon>0$ and $Y$ a topological space. A continuous map $f:X\to Y$ is called a \textbf{$(d,\epsilon)$-embedding} if for any $x_1,x_2\in X$
with $f(x_1)=f(x_2)$ we have $d(x_1,x_2)<\epsilon$. Define $$\mathrm{Widim}_\epsilon(X,d)=\min_{K\in\mathcal{K}}\dim(K),$$
where $\dim(K)$ is the Lebesgue covering dimension of the space $K$
and $\mathcal{K}$ denotes the collection of compact metrizable spaces $K$ satisfying that
there is a $(d,\epsilon)$-embedding $f:X\to K$.
Note that $\mathcal{K}$ is always nonempty since we can take
$K=X$ which is a compact metric space and $f=id$ which is the identity map from $X$ to itself.

Let $(X,\mathbb{R})$ be a flow.
For $x,y\in X$ and a subset $A$ of $\mathbb{R}$ let $$d_A(x,y)=\sup_{r\in A}d(rx,ry).$$
For $R>0$ denote by $d_R$ the metric $d_{[0,R]}$ on $X$.
Clearly, the metric $d_R$ is compatible with the topology on $X$.
\begin{Proposition}
For any $\epsilon>0$, we have

(1) $\mathrm{Widim}_\epsilon(X,d)\le\dim(X)$;

(2) if $0<\epsilon_1<\epsilon_2$ then $\mathrm{Widim}_{\epsilon_1}(X,d)\ge\mathrm{Widim}_{\epsilon_2}(X,d)$;

(3) if $0\le R_1<R_2$ then $\mathrm{Widim}_\epsilon(X,d_{R_1})\le\mathrm{Widim}_\epsilon(X,d_{R_2})$;

(4) $\mathrm{Widim}_\epsilon(X,d_{[r_1,r_2]})=\mathrm{Widim}_\epsilon(X,d_{[r_0+r_1,r_0+r_2]})$ for any $r_0,r_1,r_2\in\mathbb{R}$;

(5) $\mathrm{Widim}_\epsilon(X,d_{N+M})\le\mathrm{Widim}_\epsilon(X,d_N)+\mathrm{Widim}_\epsilon(X,d_M)$ for any $N,M\ge0$.
\end{Proposition}\begin{proof}
Since $(X,d)$ is a compact metric space that belongs to $\mathcal{K}$,
we have (1). Points (2) and (3) follow from the definition.
Let $\epsilon>0$.
If $K$ is a compact metrizable space and $f:X\to K$ is a continuous map
such that for any $x_1,x_2\in X$ with $f(x_1)=f(x_2)$ we have $d_{[r_1,r_2]}(x_1,x_2)<\epsilon$,
then $f\circ r_0:X\to K$ is a continuous map
such that for any $x_1,x_2\in X$ with $f\circ r_0(x_1)=f\circ r_0(x_2)$
we have
$d_{[r_1,r_2]}(r_0x_1,r_0x_2)<\epsilon$ which implies that
$d_{[r_0+r_1,r_0+r_2]}(x_1,x_2)<\epsilon$.
This shows (4).

To see (5),
let $\epsilon>0$,
$K$ (resp. $L$) be a compact metrizable space and $f:X\to K$ (resp. $g:X\to L$) be
a continuous map such that for any $x_1,x_2\in X$ with $f(x_1)=f(x_2)$ (resp. $g(x_1)=g(x_2)$)
we have $d_N(x_1,x_2)<\epsilon$ (resp. $d_M(x_1,x_2)<\epsilon$).
Define $F:X\to K\times L$ by $F(x)=(f(x),g(Nx))$ for every $x\in X$.
Clearly, $K\times L$ is a compact metrizable space and the map $F$ is continuous.
For $x,y\in X$,
if $F(x)=F(y)$ then $f(x)=f(y)$ and $g(Nx)=g(Ny)$, thus we have $d_N(x,y)<\epsilon$
and $d_M(Nx,Ny)<\epsilon$, and hence $d_{N+M}(x,y)<\epsilon$.
It follows that
$\mathrm{Widim}_\epsilon(X,d_{N+M})\le\dim(K\times L)\le\dim(K)+\dim(L)$.
Thus,
$\mathrm{Widim}_\epsilon(X,d_{N+M})\le\mathrm{Widim}_\epsilon(X,d_N)+\mathrm{Widim}_\epsilon(X,d_M)$.
\end{proof}We define the {\bf mean dimension} of a flow $(X,\mathbb{R})$ by:
$$\mathrm{mdim}(X,\mathbb{R})=\lim_{\epsilon\to0}\lim_{N\to\infty}\frac{\mathrm{Widim}_\epsilon(X,d_N)}{N}.$$

The limit exists by the Ornstein-Weiss lemma \cite[Theorem 6.1]{LW} as subadditivity holds.

Next we recall the definition of mean dimension for $\mathbb{Z}$-actions in \cite[Definition 2.6]{LW}.
Let $(X,T)$ be a $\mathbb{Z}$-action.
For $x,y\in X$ and $N\in\mathbb{N}$, denote
$$d^\mathbb{Z}_N(x,y)=\max_{n\in\mathbb{Z}\cap[0,N-1]}d(T^n(x),T^n(y)).$$
Define the mean dimension of $(X,T)$ by:
$$\mathrm{mdim}(X,\mathbb{Z})=\mathrm{mdim}(X,T)=\lim_{\epsilon\to0}
\lim_{N\to\infty(N\in\mathbb{N})}\frac{\mathrm{Widim}_\epsilon(X,d^\mathbb{Z}_N)}{N}.$$
\begin{Proposition}
Let $(X,\mathbb{R})$ be a flow.
If $X$ is finite dimensional
then $\mathrm{mdim}(X,\mathbb{R})=0$.
\end{Proposition}\begin{proof}
We have $\mathrm{Widim}_\epsilon(X,d_N)\le\dim(X)<+\infty$.
The result follows.\end{proof}
Although the definition of mean dimension for $\mathbb{R}$-actions
depends on the metric $d$,
the next proposition shows that the mean dimension of a flow has the same value for
all metrics compatible with the topology.
Therefore mean dimension is an invariant of $\mathbb{R}$-actions.

\begin{Proposition}
Let $(X,\mathbb{R})$ be a flow.
Suppose that $d$ and $d'$ are compatible metrics on $X$.
Then $\mathrm{mdim}(X,\mathbb{R};d)=\mathrm{mdim}(X,\mathbb{R};d')$.
\end{Proposition}

\begin{proof}
Since $d$ are $d'$ are equivalent, the identity map $id:(X,d')\to(X,d)$ is uniformly continuous.
Thus, for every $\epsilon>0$ there is $\delta>0$ with $\delta<\epsilon$ such that
for any $x,y\in X$ with $d'(x,y)<\delta$ we have $d(x,y)<\epsilon$
which implies that
$\mathrm{Widim}_\epsilon(X,d_N)\le\mathrm{Widim}_\delta(X,d'_N)$
for every $N\in\mathbb{N}$.
Noting that $\epsilon\to0$ yields $\delta\to0$ we obtain that
$\mathrm{mdim}(X,\mathbb{R};d)\le\mathrm{mdim}(X,\mathbb{R};d')$.
In the same way we also obtain
$\mathrm{mdim}(X,\mathbb{R};d')\le\mathrm{mdim}(X,\mathbb{R};d)$.
\end{proof}

\begin{Proposition}[{\cite[Def. 2.6]{LW}}]\label{ddz}
Let $(X,\mathbb{Z})$ be a t.d.s. If $d$ and $d'$ are compatible metrics on $X$
then we have $\mathrm{mdim}(X,\mathbb{Z};d)=\mathrm{mdim}(X,\mathbb{Z};d')$.
\end{Proposition}

Note that a flow $(X,(\varphi_r)_{r\in\mathbb{R}})$
naturally induces a ``sub-$\mathbb{Z}$-action'' $(X,\varphi_1)$.
\begin{Proposition}\label{rz}
Let $(X,(\varphi_r)_{r\in\mathbb{R}})$ be a flow.
Then $\mathrm{mdim}(X,(\varphi_r)_{r\in\mathbb{R}})=\mathrm{mdim}(X,\varphi_1)$.
\end{Proposition}\begin{proof}
Recall that for any compatible metric $D$ on $X$ and $R>0$, we denote
$D_R=D_{[0,R]}$. For a flow $(X,d;\mathbb{R})$ and $N\in\mathbb{N}$,
we have
$$(d_1)^\mathbb{Z}_N=(d^\mathbb{Z}_N)_1=d_N.$$
Thus,
$$\mathrm{mdim}(X,\mathbb{R};d)=\mathrm{mdim}(X,\mathbb{Z};d_1).$$
Since $d_1$ and $d$ are compatible metrics on $X$,
by Proposition \ref{ddz} we have
$$\mathrm{mdim}(X,\mathbb{Z};d_1)=\mathrm{mdim}(X,\mathbb{Z};d).$$
Combining the two equalities we have as desired $$\mathrm{mdim}(X,(\varphi_r)_{r\in\mathbb{R}})=\mathrm{mdim}(X,\varphi_1).$$
\end{proof}Thus if the space is not metrizable then we may take
$\mathrm{mdim}(X,\varphi_1)$ as the definition of mean dimension.
\begin{Proposition}
Let $(X,(\varphi_r)_{r\in\mathbb{R}})$ be a flow.
If the topological entropy of $(X,(\varphi_r)_{r\in\mathbb{R}})$ is finite
then the mean dimension of $(X,(\varphi_r)_{r\in\mathbb{R}})$ is zero.
\end{Proposition}\begin{proof}
By \cite[Proposition 8.3.6]{HK} we have
$h_{top}(X,\varphi_1)=h_{top}(X,(\varphi_r)_{r\in\mathbb{R}})$ which is finite.
By \cite[Theorem 4.2]{LW} we have $\mathrm{mdim}(X,\varphi_1)=0$.
By Proposition \ref{rz}, $\mathrm{mdim}(X,(\varphi_r)_{r\in\mathbb{R}})=0$.
\end{proof}The following proposition directly follows from the definition.
\begin{Proposition}
For any flow $(X,(\varphi_r)_{r\in\mathbb{R}})$ and $c\in\mathbb{R}$,
$$\mathrm{mdim}(X,(\varphi_{cr})_{r\in\mathbb{R}})=|c|\cdot\mathrm{mdim}(X,(\varphi_r)_{r\in\mathbb{R}}).$$
\end{Proposition}
\section{Construction of minimal real flows with arbitrary mean dimension}\label{sec:construction}
By defintion $\mathrm{mdim}(X,\mathbb{R})$ belongs to $[0,+\infty]$.
In this section we will show that for every $r\in[0,+\infty]$,
there is a minimal flow $(X,\mathbb{R})$ with $\mathrm{mdim}(X,\mathbb{R})=r$.

Recall that there are natural constructions for passing from a $\mathbb{Z}$-action to a flow,
and vice versa \cite[Section 1.11]{BS}.
Let $(X,T)$ be a $\mathbb{Z}$-action and $f:X\to(0,\infty)$ be a continuous function (in particular bounded away from $0$).
Consider the quotient space (equipped with the quotient topology) $$S_fX=\{(x,t)\in X\times\mathbb{R}^+:0\le t\le f(x)\}/\sim,$$
where $\sim$ is the equivalence relation $(x,f(x))\sim(Tx,0)$.
The {\bf suspension} over $(X,T)$ generated by the {\bf roof function} $f$ is
the flow $(S_fX,(\psi_t)_{t\in\mathbb{R}})$ given by
$$\psi_t(x,s)=(T^nx,s') \text{ for } t\in\mathbb{R} \text{ and } (x,s)\in S_fX,$$
where $n$ and $s'$ satisfy
$$\sum_{i=0}^{n-1}f(T^ix)+s'=t+s, \,\; 0\le s'\le f(T^nx).$$
In other words, flow along $\{x\}\times\mathbb{R}^+$ to $(x,f(x))$ then continue from
$(Tx,0)$ (which is the same as $(x,f(x))$) along $\{Tx\}\times\mathbb{R}^+$ and so on.
When $f\equiv 1$, then $S_fX$ is called the \textbf{mapping torus} over $X$.

Let $d$ be a compatible metric on $X$. Bowen and Walters introduced a compatible metric $\tilde{d}$ on $S_fX$ \cite[Section 4]{BW} known today as the \textbf{Bowen-Walters metric}\footnote{Note that in \cite{BW} it is assumed  that $\mathrm{diam}(X)<1$ but this is unnecessary.}. Let us recall the construction. First assume $f\equiv1$.
We will introduce $\tilde{d}_{S_1X}$ on the space $S_1X$. First, for $x,y\in X$ and $0\le t\le1$ define
the length of the horizontal segment $((x,t),(y,t))$ by:
$$d_h((x,t),(y,t))=(1-t)d(x,y)+td(Tx,Ty).$$
Clearly, we have $d_h((x,0),(y,0))=d(x,y)$ and $d_h((x,1),(y,1))=d(Tx,Ty)$.
Secondly, for $(x,t),(y,s)\in S_1X$ which are on the same orbit define
the length of the vertical segment $((x,t),(y,t))$ by:
$$d_v((x,t),(y,s))=\inf\{|r|:\psi_r(x,t)=(y,s)\}.$$
Finally, for any $(x,t),(y,s)\in S_1X$ define the distance
$\tilde{d}_{S_1X}((x,t),(y,s))$ to be the infimum of the lengths of paths
between $(x,t)$ and $(y,s)$ consisting of a finite number of horizontal and vertical segments. Bowen and Walters showed this construction gives rise to a compatible metric on $S_1X$.
Now assume a continuous function $f:X\to(0,\infty)$ is given. There is a natural homeomorphism $i_f:S_1X\to S_fX$ given by $(x,t)\mapsto (x,tf(x))$. Define $\tilde{d}_{S_fX}=(i_f)_{*}(\tilde{d}_{S_1X})$.

Recall from \cite[Definition 4.1]{LW} that for a $\mathbb{Z}$-action $(X,T)$,
the {\bf metric mean dimension} $\mathrm{mdim}_M(X,d)$ of $X$ with respect to
a metric $d$ compatible with the topology on $X$ is defined as follows.
Let $\epsilon>0$ and $n\in\mathbb{N}$.
A subset $S$ of $X$ is called $(\epsilon,d,n)$-spanning if for every
$x\in X$ there is $y\in S$ such that $d_n^\mathbb{Z}(x,y)\leq\epsilon$.
Set $$A(X,\epsilon,d,n)=\min\{\#S: S\subset X \text{ is } (\epsilon,d,n)\text{-spanning}\}$$
and define
$$\mathrm{mdim}_M(X,T,d)=\liminf_{\epsilon\to0}\frac{1}{|\log\epsilon|}
\limsup_{n\to\infty}\frac{1}{n}\log A(X,\epsilon,d,n).$$

Similarly one may define metric mean dimension for flows but we will not pursue this direction.
\begin{Theorem}[Lindenstrauss-Weiss {\cite[Theorem 4.2]{LW}}]\label{lw42}
For any $\mathbb{Z}$-action $(X,T)$ and any metric $d$ compatible with the topology
on $X$, $$\mathrm{mdim}(X,T)\le\mathrm{mdim}_M(X,T,d).$$
\end{Theorem}\begin{Theorem}[Lindenstrauss {\cite[Theorem 4.3]{L}}]\label{lin43}
If a $\mathbb{Z}$-action $(X,T)$ is an extension of an aperiodic minimal system
then there is a compatible metric $d$ on $X$ such that
$\mathrm{mdim}(X,T)=\mathrm{mdim}_M(X,T,d)$.
\end{Theorem}For related results we refer to \cite[Appendix A]{G17}.
\begin{Proposition}\label{suspmdim}
Let $(Y,(\varphi_r)_{r\in\mathbb{R}})$ be the mapping torus over $(X,T)$ (the suspension generated by the roof function $1$).
Assume that there is a compatible metric $d$ on $X$ with $\mathrm{mdim}_M(X,T,d)=\mathrm{mdim}(X,T)$. Then $$\mathrm{mdim}(X,T)=\mathrm{mdim}(Y,(\varphi_r)_{r\in\mathbb{R}})=\mathrm{mdim}_M(Y,T,\tilde{d}).$$
\end{Proposition}\begin{proof}
By Proposition \ref{rz} we have
$\mathrm{mdim}(Y,(\varphi_r)_{r\in\mathbb{R}})=\mathrm{mdim}(Y,\varphi_1)$.
Since $(X,T)$ is a subsystem of $(Y,T)=(Y,\varphi_1)$,
we have
$\mathrm{mdim}(X,T)\le\mathrm{mdim}(Y,\varphi_1)$.
Note that for every $r\in[0,1)$, $\varphi_r(X)$ is a $\varphi_1$-invariant closed subset of $Y$,
and $(\varphi_r(X),\varphi_1)$ can be regarded as a copy of $(X,T)$.
Let $\epsilon>0$ and $n\in\mathbb{N}$. If $d_{n+1}^\mathbb{Z}(x,y)\leq \frac\epsilon2$ and $|t-t'|\leq \frac\epsilon2$ for $0\le t,t'<1$ then $\tilde{d}_n^{\,\mathbb{Z}}((x,t),(y,t'))\leq \epsilon$. Thus it is easy to see $A(Y,\epsilon,\tilde{d},n)\le ([1/\epsilon]+1)\cdot A(X,\epsilon/2,d,n+1)$.
In particular $$\limsup_{n\to\infty}\frac{1}{n}\log A(Y,\epsilon,\tilde{d},n)\leq \limsup_{n\to\infty}\frac{1}{n}\log A(X,\epsilon/2,d,n)$$ and we obtain that $\mathrm{mdim}_M(Y,\tilde{d})\le\mathrm{mdim}_M(X,d)$.
By Theorem \ref{lw42} we know that $\mathrm{mdim}(Y,\varphi_1)\le\mathrm{mdim}_M(Y,\tilde{d})$.
Summarizing, we have
$$
\mathrm{mdim}(X,T)\le\mathrm{mdim}(Y,\varphi_1)\le\mathrm{mdim}_M(Y,\varphi_1,\tilde{d})$$$$
\le\mathrm{mdim}_M(X,T,d)=\mathrm{mdim}(X,T).
$$
This ends the proof.\end{proof}
We note that for general roof functions Proposition \ref{suspmdim} does not hold.
Indeed Masaki Tsukamoto has informed us that he has constructed
an example of a minimal topological dynamical system $(X,T)$ with compatible metric $d$ and $f\not\equiv 1:X\to(0,\infty)$ such that
$\mathrm{mdim}(X,T)=\mathrm{mdim}_M(X,d)=0$ but $\mathrm{mdim}_M(S_f X,\varphi_1,\tilde{d})>0$ (\cite{TsuPer}).
\begin{Problem}\label{prob:mdim1suspension=mdim}
Is Proposition \ref{suspmdim} always true
without assuming that there is a compatible metric $d$ on $X$
with $\mathrm{mdim}_M(X,d)=\mathrm{mdim}(X,T)$?
\end{Problem}\begin{Problem}
Is it possible to find a topological dynamical system $(X,T)$ with compatible metric $d$ and $f:X\to(0,\infty)$ such that
$\mathrm{mdim}(X,T)=0$ and $\mathrm{mdim}(S_f X,(\varphi_r)_{r\in\mathbb{R}})\neq 0.$
\end{Problem}In Proposition \ref{suspmdim},
if $(X,T)$ is minimal then $(Y,(\varphi_r)_{r\in\mathbb{R}})$ is minimal.
In particular, by Theorem \ref{lin43} we have the following:
\begin{Proposition}\label{mnm}
Suppose that $(X,T)$ is minimal and $(Y,\mathbb{R})$ is be the mapping torus over $(X,T)$ (the suspension generated by the roof function $1$). Then $(Y,\mathbb{R})$ is also minimal
and $\mathrm{mdim}(X,T)=\mathrm{mdim}(Y,\mathbb{R})$.
\end{Proposition}\begin{Proposition}
For every $c\in[0,+\infty]$ there is a minimal flow
$(X,(\varphi_r)_{r\in\mathbb{R}})$ such that
$\mathrm{mdim}(X,(\varphi_r)_{r\in\mathbb{R}})=c$.
\end{Proposition}\begin{proof}
By the $\mathbb{Z}$-version result due to Lindenstrauss and Weiss \cite[Proposition 3.5]{LW} there is
a minimal $\mathbb{Z}$-action $(Y,\mathbb{Z})$ such that $\mathrm{mdim}(Y,\mathbb{Z})=c$.
By Proposition \ref{mnm} we obtain a minimal flow $(X,\mathbb{R})$ with $\mathrm{mdim}(X,\mathbb{R})=c$.
\end{proof}

\section{An embedding conjecture}\label{sec:conj}
We now state the main embedding theorem of this paper.
We recall some necessary notions and results in Fourier analysis. A $C^{\infty}$ function $f:\mathbb{R}\to\mathbb{C}$, is said to be rapidly decreasing if there are constants $M_{n,m}>0$ such that $|f^{(m)}(x)| < M_{n,m} |x|^{-n}$ as $x\rightarrow \infty$, 
for all $n,m\in \mathbb{N}$. The space of such function is called the \textit{Schwartz space} and is denoted by $\mathcal{S}$. For $f\in \mathcal{S}$ 
the definitions of the Fourier transform and its inverse are given by:
$$
\mathcal{F}(f)(\xi)=\int_{-\infty}^\infty e^{-2\pi it\xi}f(t)dt, \,\,\,
\overline{\mathcal{F}}(f)(t)=\int_{-\infty}^\infty e^{2\pi it\xi}f(\xi)d\xi.
$$
One has $\mathcal{F}(\mathcal{S})=\mathcal{S}$, $\overline{\mathcal{F}}(\mathcal{S})=\mathcal{S}$ and for all $f\in\mathcal{S}$, $\overline{\mathcal{F}}(\mathcal{F}(f))=\mathcal{F}(\overline{\mathcal{F}}(f))=f$.
The operators
$\mathcal{F}$ and $\overline{\mathcal{F}}$ can be extended to tempered distributions in a standard way
(for details see \cite[Chapter 7]{Schwartz} and \cite[Chapters 3 \& 4]{Strichartz}).
The tempered distributions include in particular bounded continuous functions.

Let $a<b$ be real numbers.
We define $V[a,b]$ as the space of bounded continuous functions $f:\mathbb{R}\to\mathbb{C}$
satisfying $\mathrm{supp}\mathcal{F}(f)\subset[a,b]$. We denote $B_1(V[a,b])=\{f\in V[a,b]:\,||f||_{\infty}\le1\}$ and $B_1(V^{\mathbb{R}}[-a,a])=\{f\in B_1(V[-a,a]):\,f(\mathbb{R})\subset\mathbb{R}\}$.
One may show that $B_1(V[a,b])$ is a compact metric space
with respect to the distance:
$$\boldsymbol{d}(f_1,f_2)=\sum_{n=1}^\infty\frac{||f_1-f_2||_{L^\infty([-n,n])}}{2^n}.$$
This metric coincides with the standard topology of tempered distributions
(for details see \cite[Chapter 7, Section 4]{Schwartz}).
Let $\mathbb{R}=(\tau_r)_{r\in\mathbb{R}}$ act on $B_1(V[a,b])$ by the shift:
for every $r\in\mathbb{R}$ and $f\in B_1(V[a,b])$,
$(\tau_rf)(t)=f(t+r)$ for all $t\in\mathbb{R}$.
Thus we obtain a flow $(B_1(V[a,b]),\mathbb{R})$.

In \cite[Conjecture 1.2]{LT},
Lindenstrauss and Tsukamoto posed
the following conjecture:
\begin{Conjecture}\label{conj:Z}
Let $(X,T)$ be a $\mathbb{Z}$ dynamical system and $D$ an integer. For $r\in \mathbb{N}$, define $P_r(X,T)=\{x\in X:
\, rx=x\}$.
Suppose that for every $r\in\mathbb N$ it holds that
$\dim P_r(X,T)<\frac {rD}{2}$ and $\mathrm{mdim}(X,T)<\frac D2$.
Then $(X,T)$ can be embedded in the system $(([0,1]^D)^\mathbb{Z},\sigma)$.
\end{Conjecture}
By \cite[Proposition 3.3]{LW}, $\mathrm{mdim}(([0,1]^D)^\mathbb{Z},\sigma)= D$. It is not hard to see that for $r\in \mathbb{N}$, $$\dim P_r(([0,1]^D)^\mathbb{Z},\sigma)= rD.$$ Thus the above conjecture may be rephrased as if
$$\dim P_r(X,T)<\frac {\dim P_r(([0,1]^D)^\mathbb{Z},\sigma)}{2}$$ for all $r\in \mathbb{N}$ and $$\mathrm{mdim}(X,T)<\frac {\mathrm{mdim}(([0,1]^D)^\mathbb{Z},\sigma)}{2}$$ then $(X,T)\hookrightarrow(([0,1]^D)^\mathbb{Z},\sigma)$.
We expect that a similar phenomenon holds for flows where the role of $(([0,1]^D)^\mathbb{Z},\sigma)$ is played by $(B_1(V^{\mathbb{R}}[-a,a]),\mathbb{R})$. By \cite[Footnote 4]{GQT}, $\mathrm{mdim}(B_1(V^{\mathbb{R}}[-a,a]),\mathbb{R})=2a$.
For $r\in\mathbb{R}_{> 0}$ denote $$P_r(X,\mathbb{R})=\{x\in X:\,rx=x\}.$$
We now calculate $\dim P_r(B_1(V^{\mathbb{R}}[-a,a]),\mathbb{R})$.
\begin{Proposition}
Let $r>0$ then  $\dim P_r(B_1(V^\mathbb{R}[-a,a]))=2\lfloor ar\rfloor+1$.
\end{Proposition}
\begin{proof}
  Let $f\in B_1(V^\mathbb{R}[-a,a])$ with $f(x)=f(x+r)$ for all $x\in\mathbb{R}$.
In particular we have a periodic $f\in C^\infty(\mathbb{R},\mathbb{R})$, being a restriction of a holomorphic function,
and hence the Fourier series representation of $f$, $f(x)=\sum_{k=-\infty}^\infty c_ke^\frac{2\pi ikx}{r}$, converges uniformly to $f$
and $c_{-k}=\overline{c_k}$ for all $k$. Since $\mathcal{F}(f)=c_0\mathcal{F}(1)+\sum_{k=1}^\infty
c_k\mathcal{F}(e^\frac{2\pi ikt}{r})+\overline{c_k}\mathcal{F}(e^\frac{-2\pi ikt}{r})$
is supported in $[-a,a]$, we have $c_k=0$ for $|k|>ar$. Let $N=\lfloor ar \rfloor$.
Choose $x_0<x_1<x_2<\dots<x_N$ so that $e^\frac{2\pi i\cdot x_i}{r}\neq e^\frac{2\pi i\cdot x_j}{r}$ for $i\neq j$. The Vandermonde matrix formula indicates that $\det\left(e^\frac{2\pi i\cdot k x_l}{r}\right)_{l,k=0}^{N}\neq 0$.
This implies that  the functions $e^\frac{2\pi ikx}{r}$, $0\le k\le N$ are linearly independent. Thus, we conclude that $\dim P_r(B_1(V^\mathbb{R}[-a,a]))=2\lfloor ar\rfloor+1$.
\end{proof}
We now conjecture:

\begin{Conjecture}\label{conj:R}
Let $(X,\mathbb R)$ be a flow and $a>0$ a real number. Suppose that $\mathrm{mdim}(X,\mathbb{R})<a$ and for every $r\in\mathbb R$, $\dim P_r(X,\mathbb{R})<\lfloor ar\rfloor+\frac{1}{2}$. Then $(X,\mathbb R)$ can be embedded in
the flow $(B_1(V^{\mathbb{R}}[-a,a]),\mathbb{R})$.
\end{Conjecture}
\begin{Problem}
Does Conjecture \ref{conj:R} imply Conjecture \ref{conj:Z}? Does Conjecture \ref{conj:Z} imply Conjecture \ref{conj:R}?
\end{Problem}
We give a very partial answer:

\begin{Proposition}
Assume Conjecture \ref{conj:R} holds. Let $(X,T)$ be a t.d.s such that:
\begin{enumerate}[i.]
  \item $\exists D\in \mathbb N$, $\mathrm{mdim}(X,T)<\frac D2$,
  \item $\exists b\in \mathbb R$, $b<\frac D2$ and $\forall r> \frac{3}{D-2b}, \dim P_r(X,T)<br$,
  \item $\forall r\leq \frac{1}{D-2b}$, $P_r(X,T)=\emptyset$.
      \item
      $\mathrm{mdim}(S_1 X,\mathbb{R})=\mathrm{mdim}(X,T)$
\end{enumerate}
 Then $(X,T)$ can be embedded in the system $(([0,1]^D)^\mathbb{Z},\sigma)$.
\end{Proposition}
\begin{proof}
 Note that the periodic orbits of the suspension $(S_1 X,\mathbb{R})$ have positive integer lengthes and orbits of length $r\in \mathbb{N}$ in $S_1 X$ corresponds to the $r$-periodic points of $(X,T)$ so that $P_r(X,T)=\emptyset$ implies $P_r(S_1 X,\mathbb{R})=\emptyset$ and $P_r(X,T)\neq\emptyset$ implies:
$$\dim P_r(S_1 X,\mathbb{R})=\dim P_r(X,T)+1.$$
Consider the following sequence of embeddings:
$$(X,T)\stackrel{(1)}{\hookrightarrow}  (S_1 X, \psi_1)\stackrel{(2)}{\hookrightarrow} (B_1(V^{\mathbb{R}}[-c,c]),\sigma)\stackrel{(3)}{\hookrightarrow} (([-1,1]^{D})^{\mathbb{Z}},\sigma).$$
Embedding (1) is the trivial embedding from  $(X,T)$  into $(S_1 X, \psi_1)$ where $\psi_1$ is the time-$1$ map. Embedding (3) is a consequence of \cite[Lemma 2.4]{GQT} as long as $c<\frac D2$. We now justify Embedding (2). This $\mathbb{Z}$-embedding is induced from an $\mathbb{R}$-embedding $(S_1 X,\mathbb{R})\hookrightarrow (B_1(V^{\mathbb{R}}[-c,c]),\mathbb{R})$ whose existence follows from Conjecture \ref{conj:R} which we assume to hold. We need to verify the conditions appearing in Conjecture \ref{conj:R}. Let $c$ be a real number such that $\mathrm{mdim}(X,T)<c<\frac D2$. Thus $\mathrm{mdim}(S_1 X,\mathbb{R})=\mathrm{mdim}(X,T)<c$. Let $r$ be an integer such that $r> \frac{3}{D-2b}$, then $\dim P_r(X,\mathbb{R})<br+1$, whereas $\frac 12\dim P_r(B_1(V^{\mathbb{R}}[-c,c]),\shift)=\lfloor rc\rfloor+\frac 12=cr-t_r +\frac 12$, where $0\leq t_r<1$. Note $cr-t_r +\frac 12\geq br+1$ if $(c-b)r\geq\frac 32>t_r+\frac 12$, i.e if $r\geq \frac {3}{2(c-b)}$. Thus it is enough to check it for the minimal integer $r_0$ such that $r_0> \frac{3}{D-2b}=\frac{3}{2(\frac D2-b)}$. We thus choose $b<c<\frac D2$ such that $r_0\geq \frac{3}{2(c-2b)}> \frac{3}{2(\frac D2-2b)}$ and this ends the proof.
\end{proof}

\section{An embedding theorem}\label{sec:embedding}
For every $n\in\mathbb{N}$ denote by $S_n$ the circle of circumference $n!$ (identified with $[0,n!]$).
Let $\mathbb{R}$ act on
$\prod_{n\in\mathbb{N}}S_n$ as follows:
$(x_i)_i\mapsto(x_i+r\;(\text{mod }i!))_i$, $r\in\mathbb{R}$.
Define the {\bf solenoid} (\cite[V.8.15]{NS})
$$S=\{(x_n)_n\in\prod_{n\in\mathbb{N}}S_n:x_n=x_{n+1}\;(\text{mod }n!)\}.$$
It is easy to see that $(S,\mathbb{R})$ is a (minimal) flow.

The following definitions are standard: A continuous surjective map $\psi:(X,\mathbb{Z})\rightarrow (Y,\mathbb{Z})$ is called an \textbf{extension} (of t.d.s) if for all $n\in \mathbb{Z}$ and $x\in X$ it holds $\psi(n.x)=n.\psi(x)$. A continuous surjective map  $\psi:(X,\mathbb{R})\rightarrow (Y,\mathbb{R})$ is called an \textbf{extension} (of flows) if for all $r\in \mathbb{R}$ and $x\in X$ it holds $\psi(r.x)=r.\psi(x)$.

The following embedding result, which is the main result of this paper, provides a partial positive answer to Conjecture \ref{conj:R}. This result may be understood as an analog for flows of \cite[Corollary 1.8]{GT14} which states that Conjecture \ref{conj:Z} is true for any $\mathbb{Z}$-system which is an extension of an aperiodic subshift, i.e. an aperiodic subsystem of a symbolic shift $\left(\{1,2,\dots,l\}^{\mathbb{Z}},\sigma\right)$ for some $l\in \mathbb{N}$.

\begin{Theorem}\label{thm0}
Let $a<b$ be two real numbers.
If $(X,\mathbb{R})$ is an extension of
$(S,\mathbb{R})$ and $\mathrm{mdim}(X,\mathbb{R})<b-a$,
then $(X,\mathbb{R})$ can be embedded in
$(B_1(V[a,b]),\mathbb{R})$.
\end{Theorem}

\begin{cor}
Conjecture \ref{conj:R} holds for $(X,\mathbb{R})$ which is an extension of
$(S,\mathbb{R})$.
\end{cor}
\begin{proof}
Suppose $\mathrm{mdim}(X,\mathbb{R})<a$ for some $a>0$. As $(X,\mathbb{R})$  is an extension of an aperiodic system, it is aperiodic and in particular  for every $r\in\mathbb R$, $\dim P_r(X,\mathbb{R})=0$. We have to show that  $(X,\mathbb R)$ may be embedded in the flow $(B_1(V^{\mathbb{R}}[-a,a]),\mathbb{R})$. Indeed by Theorem \ref{thm0} $(X,\mathbb{R})$ may be embedded in
$(B_1(V[0,a]),\mathbb{R})$. It is now enough to notice that one has the following embedding:
$$B_1(V[0,a])\to B_1(V^{\mathbb{R}}[-a,a]),\,\;\,
\varphi\mapsto\frac{1}{2}(\varphi+\overline{\varphi}).$$
\end{proof}
Since for any flow $(X,\mathbb{R})$,
the product flow $(X\times S,\mathbb R\times\mathbb R)$
is an extension of the flow $(S,\mathbb{R})$,
the following result is a direct corollary of Theorem \ref{thm0}.
\begin{Theorem}
For every flow $(X,\mathbb{R})$ with
$\mathrm{mdim}(X,\mathbb{R})<b-a$
(where $a<b$ are real numbers)
there is an extension $(Y,\mathbb{R})$
with $\mathrm{mdim}(X,\mathbb{R})=\mathrm{mdim}(Y,\mathbb{R})$
that can be embedded in $(B_1(V[a,b]),\mathbb{R})$.
\end{Theorem}
In our proof of Theorem \ref{thm0},
the key step is to embed $(X,\mathbb R)$
in a product flow (Theorem \ref{thm71}):\begin{Theorem}\label{thm71}
Suppose that $a<b$, $\mathrm{mdim}(X,\mathbb{R})<b-a$ and
$\Phi:(X,\mathbb{R})\to (S,\mathbb{R})$ is an extension.
Then for a dense $G_\delta$ subset of $f\in C_\mathbb{R}(X,B_1(V[a,b]))$ the map
$$(f,\Phi):X\to B_1(V[a,b])\times S,\,\;\,x\mapsto(f(x),\Phi(x))$$
is an embedding.\end{Theorem}
\begin{remark}
It is possible to prove a similar theorem where $(S,\mathbb{R})$ is replaced by a solenoid defined by circles of circumference $r_n\rightarrow_{n\rightarrow \infty} \infty$  but we will not pursue this direction.
\end{remark}

The proof is given in the next section. We start by an auxiliary result:
\begin{Proposition}\label{prop:sol}
There is an embedding of $(S,\mathbb{R})$ in $(B_1(V[0,c]),\mathbb{R})$
for any $c>0$.
\end{Proposition}
 \begin{proof}
Define a continuous and $\mathbb{R}$-equivariant map $$\phi:(S,\mathbb{R})\to(B_1(V[0,c]),\mathbb{R})$$ by:
$$S\ni x=(x_n)_n\mapsto f_x(t)=
\sum_{n\ge m(c)}\frac{1}{2^n}\cdot
e^{2\pi i(t+x_n)/n!}=\sum_{n\ge m(c)}(\frac{1}{2^n}\cdot e^{\frac{2\pi i}{n!}x_n})\cdot
e^{\frac{2\pi i}{n!}t}$$
where $m(c)\in\mathbb{N}$ it taken to be sufficiently large so that the (RHS) belongs to $B_1(V[0,c])$.

Assume $f_x(t)=f_y(t)$ for some $x=(x_n)_n,y=(y_n)_n\in S$. We claim $x=y$. This implies that the map is an embedding. Indeed it is enough to show that for all $n$, $\frac{1}{2^n}\cdot e^{\frac{2\pi i}{n!}x_n}=\frac{1}{2^n}\cdot e^{\frac{2\pi i}{n!}y_n}$. This is a consequence of the following more general lemma:

\begin{Lemma}
Let $a_n$ be an absolutely summable series ($\sum |a_n|<\infty$). Let $\lambda_n$ be a pairwise distinct sequence of real numbers bounded in absolute value by $M>0$ ($|\lambda_n|\leq M$). Then $f(z)=\sum a_n e^{i\lambda_n z}$, $z\in \mathbb{C}$, defines an entire function such that $f\equiv 0$ iff $a_n=0$ for all $n$.
\end{Lemma}
\begin{proof} (Compare with the proof of \cite[Theorem I.3.1]{Mandelbrojt72})
We claim $$\lim_{T\rightarrow \infty}\frac{1}{T} \int_0^T f(t)e^{-i\lambda_m t}dt=a_m$$ for all $m$. Thus $f\equiv 0$ implies $a_m=0$ for all $m$. Indeed $$\frac{1}{T} \int_0^T f(t)e^{-i\lambda_m t}dt=\frac{1}{T} \int_0^T \sum_{n\neq m} a_n e^{i(\lambda_n-\lambda_m) t}dt+\frac{1}{T} \int_0^T a_n dt.$$ For $n\neq m$ as $\lambda_n-\lambda_m\neq 0$, we have $$\lim_{T\rightarrow \infty}\frac{1}{T}\int_0^T e^{i(\lambda_n-\lambda_m) t}dt=0.$$ As absolute summability implies one may reorder the limiting operations one has $$\lim_{T\rightarrow \infty}\frac{1}{T} \int_0^T \sum_{n\neq m} a_n e^{i(\lambda_n-\lambda_m) t}dt= \sum_{n\neq m} \lim_{T\rightarrow \infty}\frac{1}{T} \int_0^T a_n e^{i(\lambda_n-\lambda_m) t}dt=0,$$ This  completes the proof.
\end{proof}
\end{proof}
Now we show Theorem \ref{thm0} assuming Theorem \ref{thm71}.
\begin{proof}[Proof of Theorem \ref{thm0}
assuming Theorem \ref{thm71}]
We take $a<c_1<c_2<b$ with $\mathrm{mdim}(X,\mathbb{R})<c_1-a$.
By Theorem \ref{thm71}, $(X,\mathbb{R})$ can be embedded in
$(B_1(V[a,c_1])\times S,\mathbb{R}\times\mathbb{R})$,
which, by Proposition \ref{prop:sol}, can be embedded in
$(B_1(V[a,c_1])\times B_1(V[c_2,b]),\mathbb{R}\times\mathbb{R})$,
and finally embedded in $(B_1(V[a,b]),\mathbb{R})$ by the following embedding:
$$B_1(V[a,c_1])\times B_1(V[c_2,b])\to B_1(V[a,b]),\,\;\,
(\varphi_1,\varphi_2)\mapsto\frac{1}{2}(\varphi_1+\varphi_2).$$
This ends the proof.\end{proof}\section{Embedding in a product}\label{sec:prod embed}
Let $C_\mathbb{R}(X,B_1(V[a,b]))$ be
the space of $\mathbb{R}$-equivariant continuous
maps $f:X\to B_1(V[a,b])$.
This space is nonempty because it contains the constant $0$.
The metric on $C_\mathbb{R}(X,B_1(V[a,b]))$ is chosen to be the uniform distance
$\sup_{x\in X}\boldsymbol{d}(f(x),g(x))$.
This space is completely metrizable and hence is a Baire space
(see \cite[Theorem 48.2]{M}).

We denote by $d$ the metric on $X$.
To prove Theorem \ref{thm71},
it suffices to show that the set
$$\bigcap_{n=1}^\infty\big\{f\in C_\mathbb{R}(X,B_1(V[a,b])):(f,\Phi)\text{ is a }
\frac1n\text{-embedding with respect to }d\big\}$$
is a dense $G_\delta$ subset of $C_\mathbb{R}(X,B_1(V[a,b]))$.
It is obviously a $G_\delta$ subset of $C_\mathbb{R}(X,B_1(V[a,b]))$. Therefore it remains to prove the following:
\begin{Proposition}\label{ppppp}
For any $\delta>0$ and $f\in C_\mathbb{R}(X,B_1(V[a,b]))$,
there is $g\in C_\mathbb{R}(X,B_1(V[a,b]))$ such that:

(1) for all $x\in X$ and $t\in\mathbb{R}$, $|f(x)(t)-g(x)(t)|<\delta$;

(2) $(g,\Phi):X\to B_1(V[a,b])\times S$ is a $\delta$-embedding with respect to $d$.
\end{Proposition}

\medskip

To show Proposition \ref{ppppp},
we prove several auxiliary results.
We start by quoting
\cite[Lemma 2.1]{GT14}:
\begin{Lemma}\label{11111}
Let $(X,d')$ be a compact metric space, and let $F:X\to[-1,1]^M$
be a continuous map.
Suppose that positive numbers $\delta'$ and $\epsilon$ satisfy
the following condition:
\begin{equation}\label{eq:lemma42}
d'(x,y)<\epsilon\Longrightarrow
||F(x)-F(y)||_{\infty}<\delta',
\end{equation}
then if $\mathrm{Widim}_\epsilon(X,d')<M/2$
then there is an $\epsilon$-embedding
$G:X\to[-1,1]^M$ satisfying:
$$\sup_{x\in X}||F(x)-G(x)||_{\infty}<\delta'.$$
\end{Lemma}
We say that a holomorphic  function $g$ in $S\subset\mathbb{C}$ is of \textbf{exponential type} if for all $z\in S$, $|g(z)|\leq C e^{T|z|}$ for some $C,T>0$. The following classical theorem is proven in \cite[Section 3.1.7]{DM}.
\begin{Theorem}[Phragm\'{e}n--Lindel\"{o}f principle]\label{thm:pl}
Let $g$ be a function of exponential type that is holomorphic in the sector
$$S=\left\{z\in \mathbb{C}\,{\big |}\,\alpha <\arg z<\beta \right\}$$ of angle $\beta -\alpha <\pi$, and continuous on its boundary. If $|g(z)|\leq 1$ for $z\in \partial S$ then $|g(z)|\leq 1$ for $z\in S$.
\end{Theorem}

According to the classical Paley-Wiener theorem (\cite[Theorem 19.3]{Rud87}), if $f\in L^2(\mathbb{R})$ extends to an entire function $F$ such that there exist $A,C>0$ such that for all $z=x+iy\in\mathbb{C}$, 
$ |f(x+yi)| \leq C e^{2\pi A|y|}$, then $\mathcal{F}(f)\in L^2(\mathbb{R})$ is supported in $[-A,A]$. We will need a generalized version: 
\begin{Theorem}\label{thm:gpw}
Let $f\in L^{\infty}(\mathbb{R})$ be a function which extends to an entire function $F:\mathbb{C}\rightarrow\mathbb{C}$ ($F_{|\mathbb{R}}=f$) such that there exist $A,C>0$ and $M\in\mathbb{N}$ such that for all $z=x+iy\in\mathbb{C}$ 
$$|F(z)|\leq C(1+|z|)^{M}\cdot e^{2\pi A|y|}.$$
Then $f\in V[-A,A]$. 
\end{Theorem}
\begin{proof}
See\footnote{While reading the proof in the reference one should note that in \cite{Strichartz} the Fourier transform is defined as $\mathcal{F}(f)(\xi)=\int_{-\infty}^\infty e^{it\xi}f(t)dt$.} \cite[Theorem 7.2.3]{Strichartz}.
\end{proof}
Let $\rho>0$ and $N\in\mathbb N$ so that
$\rho N!\in\mathbb N$. Define:
$$L(\rho)=\{\frac k\rho\}_{k\in\mathbb{Z}},\,\, L^*(\rho)=L(\rho)\setminus\{0\}.$$

 In the next lemma we write $x\lesssim y$ for two real numbers $x$ and $y$ if there exists a constant $C>0$ which depends only on $\rho$ and $N$ such that $x\le Cy$.
\begin{Lemma}\label{lem:prod_function}
Let $$f(z)=\lim_{A\to\infty}\prod_{\lambda\in L(\rho),0<|\lambda|<A}\left(1-\frac{z}{\lambda}\right).$$ Then $f$
defines a holomorphic function in $\mathbb{C}$ satisfying
$$f(0)=1,\quad\quad f(\lambda)=0,\quad\forall\;\lambda\in L^*(\rho).$$
Moreover, for all $z\in\mathbb{C}$ we have
$$|f(z)|\lesssim(1+|z|)^{5\rho N!}\cdot e^{\pi\rho|y|},$$
where $y$ is the imaginary part of $z$.
\end{Lemma}
\begin{proof}
We first show the convergence of $f(z)$. Notice
$$f(z)=\lim_{A\to\infty}\prod_{\lambda\in L(\rho),0<\lambda<A}\left(1-\frac{z^2}{\lambda^2}\right)$$
As $\sum_{\lambda\in L(\rho),0<\lambda}\frac{1}{\lambda^2}$ converges, the limit above converges locally uniformly (see \cite[\S 29, Theorems 6 \& 7]{kno51}. Thus, $f(z)$ is a holomorphic function which satisfies
$$f(0)=1,\quad\quad f(\lambda)=0,\quad\forall\;\lambda\in L^*(\rho).$$
Next we shall estimate the growth of $f$ on the real line. Suppose $x>0$ and let $k$ be the integer with $kN!\le x<(k+1)N!$. We may assume $k>0$, as the case $k=0$ is easier and can be dealt with in a similar way. For $n\in\mathbb{Z}$, set
$$L_n=L(\rho)\cap[nN!,(n+1)N!).$$
For $\lambda\in L_n$ with $n\le-2$ or $n\ge k+1$ we have
$$|1-x/\lambda|\le1-x/(n+1)N!$$
and hence
$$\prod_{\lambda\in L_n}\left|1-\frac{x}{\lambda}\right|\le\left|1-\frac{x}{(n+1)N!}\right|^{\rho N!}.$$
For $\lambda\in L_n$ with $1\le n<k$ we have
$$|1-x/\lambda|\le x/(nN!)-1$$
and hence
$$\prod_{\lambda\in L_n}\left|1-\frac{x}{\lambda}\right|\le\left|1-\frac{x}{nN!}\right|^{\rho N!}.$$
The factors for $n=-1,0,k$ need to be treated separately. Recall Euler's sine product formula (\cite{C15}):
$$\frac{\sin z}{z}=\lim_{A\to\infty}\prod_{0<|n|<A}\left(1-\frac{z}{n\pi}\right)$$
Using this it is easy to see that $|f(x)|$ is bounded by
\begin{align*}
&\prod_{0\ne\lambda\in L_{-1}\cup L_0\cup L_k}\left|1-\frac{x}{\lambda}\right|\cdot\lim_{A\to\infty}\prod_{|n|<A,n\ne0,k,k+1}\left|1-\frac{x}{nN!}\right|^{\rho N!}\\
=&\prod_{0\ne\lambda\in L_{-1}\cup L_0\cup L_k}\left|1-\frac{x}{\lambda}\right|\cdot\left|\frac{\sin\frac{\pi x}{N!}}{\frac{\pi x}{N!}\left(1-\frac{x}{kN!}\right)\left(1-\frac{x}{(k+1)N!}\right)}\right|^{\rho N!}.
\end{align*}
The first factor is easy to estimate:
$$\prod_{0\neq \lambda\in L_{-1}\cup L_0\cup L_k}\left|1-\frac{x}{\lambda}\right|\lesssim(1+x)^{3\rho N!}.$$
Set $t=x/N!$,
$$\frac{\sin\frac{\pi x}{N!}}{\frac{\pi x}{N!}\left(1-\frac{x}{kN!}\right)\left(1-\frac{x}{(k+1)N!}\right)}=\frac{k(k+1)\sin\pi t}{\pi t(k-t)(k+1-t)}.$$
By the mean value theorem,
$$\left|\frac{\sin\pi t}{t}\right|\le\pi,\quad\left|\frac{\sin\pi t}{k-t}\right|\le\pi,\quad\left|\frac{\sin\pi t}{k+1-t}\right|\le\pi.$$
Thus,
$$\left|\frac{k(k+1)\sin\pi t}{\pi t(k-t)(k+1-t)}\right|\lesssim k(k+1)\lesssim(1+x)^2.$$
Therefore
$$|f(x)|\lesssim(1+x)^{5\rho N!}.$$
The case $x<0$ is similar so we get
$$|f(x)|\lesssim(1+|x|)^{5\rho N!}.$$
We now turn to estimating $|f(yi)|$ for $y\in \mathbb{R}\setminus\{0\}$. For $r>0$ we set
$$n(r)=\#(L^*(\rho)\cap(-r,r)).$$
We have
$$n(r)< 2\rho r,$$
Note that for $0<r\leq \frac{1}{\rho}$, one has $n(r)=0$. Since $$|f(yi)|^2=\prod_{\lambda\in L^*(\rho)}(1+y^2/\lambda^2),$$
As $n(r)$ is monotonic increasing, we may use the Riemann–Stieltjes integral to write:
$$\log|f(yi)|=\frac{1}{2}\sum_{\lambda\in L^*(\rho)}\log\left(1+\frac{y^2}{\lambda^2}\right)=\frac{1}{2}\int_{\frac{1}{\rho}}^\infty\log\left(1+\frac{y^2}{r^2}\right)dn(r).$$
Using integration by parts for the Riemann–Stieltjes integral (\cite[Theorem 12.14]{G94}), we see that for all $R\geq \frac{1}{\rho}$ it holds:
$$\frac{1}{2}\int_{\frac{1}{\rho}}^{R}\log\left(1+\frac{y^2}{r^2}\right)dn(r)=\frac{1}{2}\left( \log\left(1+\frac{y^2}{r^2}\right)n(r)\bigg|_{\frac{1}{\rho}}^{R}-\int_{\frac{1}{\rho}}^{R}n(r)d\,\log\left(1+\frac{y^2}{r^2}\right)\right).$$
Taking $R\rightarrow \infty$, we conclude:
$$\log|f(yi)|=y^2\int_{\frac{1}{\rho}}^\infty \frac{n(r)}{r(r^2+y^2)}dr$$
Since $n(r)\le 2\rho r$, we deduce
$$\log|f(yi)|\le 2\rho y^2\int_{\frac{1}{\rho}}^\infty\frac{dr}{r^2+y^2}.$$
It is a standard exercise to show:
$$\int_{0}^\infty\frac{dr}{r^2+y^2}=\frac{1}{|y|}\int_{0}^\infty\frac{dr}{1+r^2}=\frac{\pi}{2|y|}.$$
It follows that
$$|f(yi)|\leq e^{\pi\rho |y|}.$$
Finally we show that $|f(z)|$ grows at most exponentially. Let $z=x+yi$. We may assume $x,y>0$, as all the other cases are similar. Let $k$ be the integer with $kN!\le x<(k+1)N!$. Set
$$L'=L(\rho)\setminus(L_{k-1}\cup L_k\cup L_{k+1}).$$
We estimate
$$\prod_{0\ne\lambda\in L_{k-1}\cup L_k\cup L_{k+1}}\left|1-\frac{z}{\lambda}\right|\lesssim(1+|z|)^{3\rho N!}.$$$$\lim_{A\to\infty}\prod_{0\ne\lambda\in L',|\lambda|<A}\left|1-\frac{z}{\lambda}\right|^2=\lim_{A\to\infty}\prod_{0\ne\lambda\in L',|\lambda|<A}\left\{\left(1-\frac{x}{\lambda}\right)^2+\frac{y^2}{\lambda^2}\right\}$$$$=\left\{\lim_{A\to\infty}\prod_{0\ne\lambda\in L',|\lambda|<A}\left(1-\frac{x}{\lambda}\right)^2\right\}\cdot\prod_{0\ne\lambda\in L'}\left\{1+\frac{y^2}{(\lambda-x)^2}\right\}.$$
As in the proof of $|f(x)|\lesssim(1+|x|)^{5\rho N!}$ we estimate
$$\lim_{A\to\infty}\prod_{0\ne\lambda\in L',|\lambda|<A}\left(1-\frac{x}{\lambda}\right)^2\lesssim(1+x)^{12\rho N!}.$$
As in $|f(yi)|\leq e^{\pi\rho|y|}$,
$$\prod_{0\ne\lambda\in L'}\left\{1+\frac{y^2}{(\lambda-x)^2}\right\}\leq e^{2\pi\rho|y|}.$$
Thus, we deduce that $|f(z)|$ grows at most exponentially.

We have thus shown that $f(z)$ has exponential type and satisfies $|f(x)|\lesssim(1+|x|)^{5\rho N!}$ and $|f(yi)|\leq e^{\pi\rho|y|}$. By the Phragm\'en--Lindel\"of principle of Theorem \ref{thm:pl} (e.g. in the first quadrant $x,y\ge0$) applied  to $(1+z)^{-5\rho N!}e^{\pi\rho iz}f(z)$), the claim follows.
\end{proof}

Next we construct an interpolation function based on \cite[pp. 351--365]{Beurling}:
\begin{Proposition}\label{prop:interpolation}
Let $a<b$. Let $\rho>0$ with $\rho\in\mathbb Q$ and $\rho<b-a$. There exists
$\varphi\in V[a,b]$ rapidly decreasing so that $\varphi(0)=1$ and for all $\lambda\in L^*(\rho)$,
$\varphi(\lambda)=0$.
\end{Proposition}\begin{proof}
Fix $\tau>0$ so that $\rho+\tau<b-a$. Let $\psi(\xi)\in \mathcal{S}$ be a nonnegative smooth function in $\mathbb{R}$ satisfying
\[  \supp(\psi)\subset \left[-\frac{\tau}{2}, \frac{\tau}{2}\right], \quad \int_{-\infty}^\infty \psi(\xi)d\xi = 1.\]
Define the function $h:\mathbb{C}\rightarrow \mathbb{C}$ by $$h(z)=\int_{-\frac{\tau}{2}}^{\frac{\tau}{2}} \psi(\xi)e^{2\pi i z \xi}d\xi.$$
It is easy to see that $h$ is  an entire function which satisfies:
\begin{equation}\label{eq:h}
h_{|\mathbb{R}}=\overline{\mathcal{F}}(\psi)\in \mathcal{S}, \quad h(0)=1,\quad
\left|h(x+yi)\right|\le e^{\pi\tau|y|},
\quad\forall x,y\in\mathbb{R}.
\end{equation} 
Let $$g(z)=\lim_{A\to\infty}  \prod_{\lambda\in L(\rho), 0<|\lambda|<A}\left(1-\frac{t}{\lambda}\right).$$
By Lemma \ref{lem:prod_function}, $g(z)$ is an entire function. Thus we may define the following entire functions:
$$\tilde{\varphi}(z)=h(z)g(z),\,\, \varphi(z)=e^{\pi iz(a+b)}\tilde{\varphi}(z).$$
It is easy to see that $\varphi(0)=1$ and for all $\lambda\in L^*(\rho)$, $\varphi(\lambda)=0$. By Lemma \ref{lem:prod_function}, $g_{|\mathbb{R}}$ has polynomial growth. Therefore as $\overline{\mathcal{F}}(\psi)$ is rapidly decreasing, so are $\varphi_{|\mathbb{R}}$ and $\tilde{\varphi}_{|\mathbb{R}}$. By Lemma \ref{lem:prod_function} and (\ref{eq:h}) (recall the convention $z=x+iy$):
$$|\tilde{\varphi}(z)|\lesssim(1+|z|)^{5\rho N!}\cdot e^{\pi(\rho+\tau)|y|}$$
As in addition $\tilde{\varphi}_{|\mathbb{R}}$ is bounded (as it is rapidly decreasing), it follows from Theorem \ref{thm:gpw} that 
 $\tilde{\varphi}\in V[\frac{-\rho-\tau}{2},\frac{\rho+\tau}{2}]\subset V[\frac{a-b}{2},\frac{b-a}{2}]$. This immediately implies $\varphi\in V[a,b]$ which finishes the proof.
\end{proof}

Now we are ready to prove Proposition \ref{ppppp}.\begin{proof}[Proof of Proposition \ref{ppppp}]
We take $\delta>0$ and $f\in C_\mathbb{R}(X,B_1(V[a,b]))$.
Without loss of generality,
we assume that $|f(x)(t)|\le1-\delta$ for all $x\in X$ and $t\in\mathbb{R}$
(by replacing $f$ with $(1-\delta)f$ if necessary).
Fix $\rho\in\mathbb Q$ with
$$\mathrm{mdim}(X,\mathbb R)<\rho<b-a.$$ Let $\varphi$ be the function constructed in Proposition \ref{prop:interpolation}. As $\varphi$ is a rapidly decreasing
function, we may find $K>0$ such that:
\begin{equation}\label{eq75}
|\varphi(t)|\le\frac{K}{1+|t|^2}.
\end{equation}Let $\delta'>0$ be such that:
\begin{equation}\label{eq71}
\delta'\cdot\sum_{\lambda\in L(\rho)}\frac{K}{1+|t-\lambda|^2}<\delta
\text{ for all }t\in\mathbb{R}.
\end{equation}Fix $\epsilon\in(0,\delta)$. Let $N\in\mathbb N$ be such that
$\rho N!\in\mathbb N$, $\mathrm{Widim}_\epsilon(X,d_{N!})<\rho N!$, and such that
\begin{equation}\label{eq72}
d_{N!}(x,y)<\epsilon\text{ implies }
|f(x)(t)-f(y)(t)|<\frac{\delta'}{2}
\text{ for all }t\in[0,N!].
\end{equation}Define:
$$F:X\to[0,1]^{2\rho N!}=([0,1]^{2})^{\rho N!},\;\;\;
F(x)=(\text{Re}f(x)|_{L(\rho,N)},\text{Im}f(x)|_{L(\rho,N)}).$$
$$F^{\mathbb{C}}:X\to\mathbb{C}^{\rho N!},\;\;\; F^{\mathbb{C}}(x)=f(x)|_{L(\rho,N)}.$$
 Let $M=2\rho N!$, $d'=d_{N!}$.  Equation \eqref{eq72} implies that Equation \eqref{eq:lemma42} holds, so Lemma \ref{11111} implies,
there is an $(d_{N!},\epsilon)$-embedding $G:X\to[-1,1]^{2\rho N!}$ such that $\sup_{x\in X}||F(x)-G(x)||_{\infty}<\frac{\delta'}{2}$.
Similarly to $F^{\mathbb{C}}(x)(k)$, we introduce the notation $G^{\mathbb{C}}(x)(k)$, $k=0,\ldots,\rho N!-1$ in the natural way. Notice it holds:
\begin{equation}\label{deltaprime}
\sup_{x\in X}||F^{\mathbb{C}}(x)-G^{\mathbb{C}}(x)||_{\infty}<\delta'.\end{equation}
Take $x\in X$. Denote $\Phi(x)=(\Phi(x)_n)_{n\in\mathbb N}$,
where $\Phi(x)_n\in S_{n!}$.
For every $n\in\mathbb Z$
let $$\Lambda(x,n)=nN!-\Phi(x)_N+L(\rho,N),$$
$$\Lambda(x)=\bigcup_{n\in\mathbb{Z}}\Lambda(x,n)\subset\mathbb R.$$
\begin{center}\includegraphics[width=10.5cm]{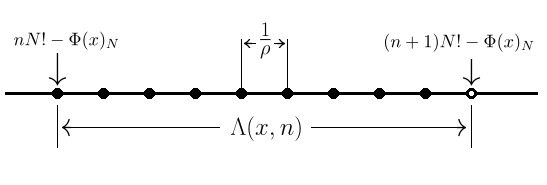}\\Figure 6.1. The set $\Lambda(x,n)$.\end{center}
Next we construct a perturbation $g$ of $f$:
$$g(x)(t)=f(x)(t)+h(x)(t),$$
where $h(x)(t)$ is defined by
$$
\sum_{n\in\mathbb Z}
\sum_{k=0}^{\rho N!-1}
\big(G^{\mathbb{C}}(T^{nN!-\Phi(x)_N}x)(k)-F^{\mathbb{C}}(T^{nN!-\Phi(x)_N}x)(k)\big)
\varphi(t-(\frac k\rho+nN!-\Phi(x)_N))
.$$
As $\varphi$ is rapidly decreasing the sum defining $g(x)$ for fixed $x$ converges in the compact open topology to a function in  $V[a,b]$. Moreover the mapping $x\mapsto g(x)$ is continuous. In order to see that $g(x)$ is $\mathbb{R}$-equivariant, it suffices to deal with $h(x)$ (because $f$ is already $\mathbb{R}$-equivariant). To see that $h(x)$ is $\mathbb{R}$-equivariant we first note that for $0\leq r<N!-\Phi(x)_N$ we have $\Phi(T^rx)_N=\Phi(x)_N+r$ and hence from the definition of $h$ it follows that $h(T^rx)(t)=h(x)(t+r)$. Similarly, if $N!-\Phi(x)_N\leq r<N!$ then $\Phi(T^rx)_N=r-(N!-\Phi(x)_N)$ and hence $(T^{nN!-\Phi(x)_N}T^rx)(k)=(T^{(n+1)N!-\Phi(x)_N-r}T^rx)(k)$. Using such information in each summand in the sum over $k$'s appearing in the definition of $h(T^rx)(t)$, and then substituting $n+1$ by $n$ when summing over $n\in\mathbb{Z}$, we get as desired $h(T^rx)(t)=h(x)(t+r)$ for $r$'s in this range. If $r=sN!$ where $s\in\mathbb{Z}$ then $\Phi(T^rx)_N=r-(sN!-\Phi(x)_N)$ and hence $(T^{nN!-\Phi(T^rx)_N}T^rx)(k)=(T^{(n+s)N!-\Phi(x)_N}x)(k)$. Using this information in each summand in the sum over $k$'s appearing in the definition of $h(T^rx)(t)$, and substituting $n+s$ by $n$ when summing over $n\in\mathbb{Z}$, we obtain as desired $h(T^rx)(t)=h(x)(t+r)$ for $r$'s in this range. Finally if $r=sN!+r'$ where $s\in\mathbb{Z}$ and $0<r'<N!$ we use the additivity properties of the terms involved in order to combine the two cases and get the desired result.
Note that by Equations \eqref{eq75} and \eqref{eq71}  for all $x\in X$ and $t\in \mathbb{R}$:
$$\sum_{n\in\mathbb Z}
\sum_{k=0}^{\rho N!-1}
\varphi(t-(\frac k\rho+nN!-\Phi(x)_N))<\frac{\delta}{\delta'}.$$
By Equation \eqref{deltaprime} for all $x\in X$, $k=0,\ldots,\rho N!-1$:
$$|G^{\mathbb{C}}(T^{nN!-\Phi(x)_N}x)(k)-F^{\mathbb{C}}(T^{nN!-\Phi(x)_N}x)(k)|<\delta'.$$
Combining the two last inequalities we have
$|g(x)(t)-f(x)(t)|<\delta$
for all  $x\in X$ and $t\in\mathbb{R}$.
Since $|f(x)(t)|\le1-\delta$,
we have $g(x)\in B_1(V[a,b])$.
Thus,
$g\in C_\mathbb{R}(X,B_1(V[a,b]))$.
It remains to check that the map
$$(g,\Phi):X\to B_1(V[a,b])\times S,\;\;\;x\mapsto(g(x),\Phi(x))$$
is a $\delta$-embedding with respect to $d$.
We take $x,x'\in X$ with $(g(x),\Phi(x))=(g(x'),\Phi(x'))$. We calculate for $k=0,\ldots,\rho N!-1$:
$$g(x)(-\Phi(x)_N+\frac k\rho)=f(x)(-\Phi(x)_N+\frac k\rho)+
\big(G^{\mathbb{C}}(T^{-\Phi(x)_N}x)(k)-F^{\mathbb{C}}(T^{-\Phi(x)_N}x)(k)\big)
.$$
As $F^{\mathbb{C}}(T^{-\Phi(x)_N}x)(k) =f(T^{-\Phi(x)_N}x)(\frac k\rho)=f(x)(-\Phi(x)_N+\frac k\rho)$, we conclude for $k=0,\ldots,\rho N!~-~1$ that $g(x)(-\Phi(x)_N+\frac k\rho)=G^{\mathbb{C}}(T^{-\Phi(x)_N}x)(k)$. Similarly $g(x')(-\Phi(x')_N+\frac k\rho)=G^{\mathbb{C}}(T^{-\Phi(x')_N}x')(k)$. Thus:
$$g(x)(-\Phi(x)_N+\frac k\rho)=g(x')(-\Phi(x)_N+\frac k\rho)=g(x')(-\Phi(x')_N+\frac k\rho)$$
implies
$$G^{\mathbb{C}}(T^{-\Phi(x)_N}x)(k)=G^{\mathbb{C}}(T^{-\Phi(x')_N}x')(k)=G^{\mathbb{C}}(T^{-\Phi(x)_N}x')(k).$$
Since
$G^{\mathbb{C}}:X\to[0,1]^{\rho N!}$
is an $(d_{N!},\epsilon)$-embedding,
we have
$$d_{N!}(T^{-\Phi(x)_N}x,T^{-\Phi(x)_N}x')<\epsilon<\delta$$ which implies $d(x,x')<\epsilon<\delta$.
This ends the proof.
\end{proof}

\bigskip

\end{document}